\documentclass[11pt,twoside]{article}
\usepackage{latexsym, amssymb,amsthm,amsmath,lscape,epsfig,setspace,tikz,slashed}
\usepackage[all]{xy}
\usepackage[retainorgcmds]{IEEEtrantools}
\usetikzlibrary{decorations.markings}
 \usetikzlibrary{arrows}
\usepackage{hyperref}
\usepackage[margin=10pt,font=small,labelfont=bf]{caption}

\theoremstyle{plain}
\newtheorem{theorem}{Theorem}[section]

\theoremstyle{definition}
\newtheorem{definition}[theorem]{Definition}

\newtheorem{lemma}[theorem]{Lemma}
\newtheorem{corollary}[theorem]{Corollary}

\newenvironment{renumerate}%
{%
\begin{enumerate}}%
{\end{enumerate}%
}%

\makeatletter
\def\Ddots{\mathinner{\mkern1mu\raise\p@
\vbox{\kern7\p@\hbox{.}}\mkern2mu
\raise4\p@\hbox{.}\mkern2mu\raise7\p@\hbox{.}\mkern1mu}}
\makeatother

{\vskip6pt%
\noindent%
{\it Remark.}}%
{\vskip6pt}

{\vskip6pt%
\noindent%
{\it Remarks}. %
\begin{renumerate}}%
{\end{renumerate}\vskip6pt}

\newcommand{\R}{\text{${\mathbb R}$}}
\newcommand{\C}{\text{$\mathbb C$}}

\newcommand{\Z}{\text{$\mathbb Z$}}

\newcommand{\T}{\text{$\mathbb{T}$}}
\newcommand{\TM}{\text{$\mathbb{T}M$}}
\renewcommand{\frak}[1]{\text{$\mathfrak{#1}$}}

\newcommand{\V}{\text{$\mathbb{V}$}}
\newcommand{\J}{\text{$\mathcal{J}$}}

\newcommand{\JJ}{\text{$\mathbb{J}$}}

\newcommand{\G}{\text{$\mathcal{G}$}}
\newcommand{\U}{\text{$\mathcal{U}$}}

\newcommand{\Spin}[1]{\mathrm{Spin}({#1})}

\newcommand{\spin}[1]{\mathfrak{spin}({#1})}

\newcommand{\Rho}{P}
\renewcommand{\bar}{\overline}

\newcommand{\gf}{\text{$\varphi$}}

\newcommand{\Id}{\mathrm{Id}}

\newcommand{\deltabar}{\text{$\overline{\delta}$}}
\newcommand{\tensor}{\otimes}

\newcommand{\mc}[1]{\text{$\mathcal{#1}$}}

\newcommand{\into}{\longrightarrow}
\newcommand{\noqed}{\let\qed\relax}

\newcommand{\IP}[1]{\langle #1 \rangle}

\newcommand{\nhood}{neighbourhood}

\newcommand{\gcs}{generalized complex structure}
\newcommand{\gacs}{generalized almost complex structure}
\newcommand{\gcss}{generalized complex structures}

\newcommand{\gc}{generalized complex}

\newcommand{\gk}{generalized K\"ahler}
\newcommand{\gks}{generalized K\"ahler structure}
\newcommand{\gkss}{generalized K\"ahler structures}
\newcommand{\gkm}{generalized K\"ahler manifold}
\newcommand{\gkms}{generalized K\"ahler manifolds}

\newcommand{\wrt}{with respect to}

\newcommand{\Cour}[1]{[\![#1]\!]}

\date{} \usepackage{color} \definecolor{tocolor}{rgb}{.1,.1,.5}
\definecolor{urlcolor}{rgb}{.2,.2,.6}
\definecolor{linkcolor}{rgb}{.1,.1,.6}
\definecolor{citecolor}{rgb}{.6,.2,.1}
\hypersetup{backref=true, colorlinks=true, urlcolor=urlcolor,
  linkcolor=linkcolor, citecolor=citecolor}

\begingroup
\hypersetup{linkcolor=blue}

\endgroup

\numberwithin{equation}{section}

\author{Gil R. Cavalcanti\footnote{Utrecht University; gil.cavalcanti@gmail.com.}}

\title{Goto's generalized K\"ahler stability theorem}
\begin{document}

\maketitle

\abstract{In these notes we give a shortened and more direct proof of Goto's generalized K\"ahler stability theorem stating that if $(\J_1,\J_2)$ is a \gks\ for which $\J_2$ is determined by a nowhere vanishing closed form, then small deformations of $\J_1$ can be coupled with small deformations of $\J_2$ so that the pair remains a \gks.}

\section*{Introduction}

Generalized K\"ahler structures were introduced in 2003 by Gualtieri \cite{gualtieri-2003} and raised immediate interest accross fields  as  Gualtieri proved that the \gk\ condition is precisely equivalent to the conditions required on the target space of a $(2,2)$-supersymmetric sigma model discovered by Gates, Hull and Ro\v cek \cite{MR776369}. Despite of the fact that the bi-Hermitian structures  of Gates, Hull and Ro\v cek had been around for 20 years, there were only a handful of known examples which were not outright K\"ahler and, in the following years, we saw a march towards finding interesting examples using \gc\ insights. 

The most successful method for constructing such examples was by deforming a usual K\"ahler structure into a generalized one. Indeed, Gualtieri showed in his thesis that a complex structure can be transformed into a (non complex) generalized complex structure by use of a holomorphic Poisson bivector. The basic idea for the \gk\ case was to use a holomorphic Poisson bivector to  deform the complex structure and show that the symplectic structure could also be deformed so that the pair remained \gk. This idea was first implemented by  Hitchin  in early 2005 \cite{MR2217300}. There, he produced two examples of such structures: one on $\C P^2$ and one on $\C P^1 \times \C P^1$. Later in the same year, these examples were extended to several toric varieties by Lin and Tolman \cite{LiTo06} using a quotient construction and, in 2006,  Hitchin extended the construction to arbitrary Poisson bivectors on del Pezzo surfaces \cite{MR2371181}. The question was then quite neatly settled in 2007, when Goto  \cite{MR2669364} proved that any small deformation of the complex structure of a compact K\"ahler manifold  can be completed to a deformation of the whole generalized K\"ahler structure. Goto's result goes beyond the search of examples. In content, it is an analogue  to Kodaira and Spencer's stability theorem of K\"ahler structures \cite{MR0115189}. Precisely, Goto showed that given a \gks\ $(\J_1,\J_2)$ on a compact manifold for which $\J_2$ is determined by a closed form, then any deformation of $\J_1$ can be completed by a deformation of $\J_2$ so that the deformed pair is still \gk.

In these notes we review the proof of Goto's theorem. While the heart the argument is still the same, we use results regarding Hodge theory of \gkms\ \cite{gualtieri-2004} more judiciously as well as  new results regarding the intrinsic torsion of a generalized almost Hermitian manifold  \cite{cavalcantiskt} to clear nearly all of the setup used by Goto and produce a much cleaner and clearer proof.

{\bf Acknowledgements}: This research was supported by a Marie Curie intra-european fellowship.

\section{Linear algebra}\label{sec:linear algebra}

Given a vector space $V^m$ we let $\V = V \oplus V^*$ be its ``double".  $\V$ is endowed with a natural symmetric pairing:
$$\IP{X+\xi,Y+\eta} =\tfrac{1}{2}(\eta(X) + \xi(Y)),\qquad X,Y \in V;~\xi,\eta \in V^*.$$

Elements of $\V$ act on $\wedge^{\bullet}V^*$ via
$$(X +\xi) \cdot \gf = i_X\gf + \xi\wedge \gf.$$
This action extends to an action of  the Clifford algebra of $\V$ making $\wedge^{\bullet} V^*$ a natural representation of the space of spinors for $Spin(\V)$.

 In particular, $\wedge^{\bullet}V^*$ comes equipped with a spin invariant pairing, the {\it Chevalley pairing}:
$$(\gf,\psi)_{Ch} = -(\gf \wedge \psi^t)_{top}, $$
where $\cdot^t$ indicates transposition, an $\R$-linear operator defined on decomposable forms by
$$(\theta_1\wedge \cdots \wedge \theta_k)^t = \theta_k \wedge \cdots \wedge \theta_1,$$
and $top$ means taking the degree $m$ component.

The group $\mathrm{Spin}(\V)$ acts on both $\mathrm{Clif}(\V)$, the Clifford algebra of $\V$, and on spinors in a compatible manner, namely, for $\gamma \in \Spin{\V}$ and $\alpha \in \mathrm{Clif}^{k}( \V)$ the action of $\gamma$ on $\alpha$ is given by Clifford conjugation
$$ \gamma_* \alpha = \gamma \alpha\gamma^{-1} \in  \mathrm{Clif}^{k}( \V).$$
And for $\gf \in   \wedge^\bullet V$ we have
\begin{equation}\label{eq:compatibility}
\gamma_* \alpha \cdot \gamma\gf = \gamma \alpha \gamma^{-1}\gamma \gf = \gamma(\alpha\gf).
\end{equation}

\begin{definition}
A {\it generalized metric} on $V$ is an automorphism $\G:\V \into \V$ which is orthogonal and self-adjoint \wrt\ the natural pairing and for which the bilinear tensor
$$\IP{\G \cdot,\cdot}:\V \tensor \V \into \R$$
is positive definite.
\end{definition}

Since $\G$ is orthogonal and self-adjoint we have  $\G^{2}= \Id$, hence  $\G$ splits $\V$ into its $\pm 1$-eigenspaces: $\V = V_+ \oplus V_-$. The projection $\pi_V:\V \into V$ gives isomorphisms $\pi:V_{\pm} \into V$.

If $V$ is endowed with an orientation, we can define a generalized {\it Hodge star operator} as follows. Since $\pi_V:V_+ \into V$ is an isomorphism, $V_+$ also inherits an orientation. Then we let $\{e_1, e_2,\cdots, e_m\}$ be a positive orthonormal basis of $V_+$, let $\star = -e_m\cdot\cdots e_2\cdot e_1 \in \mathrm{Clif}(\V)$ and define
$$\star \gf := \star \cdot \gf,$$
where $\cdot$ denotes Clifford action. With this definition, we have
\begin{equation}\label{eq:inner product}
(\gf,\star \gf)_{Ch} >0 \qquad \mbox{if}~ \gf \neq 0.
\end{equation}

 For the rest of this section we will introduce structures on $V$ which force its dimension to be even so we let $m =2n$.

\begin{definition}
A {\it \gcs} on $V$ is a complex structure on $\V$ which is orthogonal \wrt\ the natural pairing. A {\it generalized Hermitian structure} on $V$ is a pair $(\G,\J_1)$ of generalized metric and  \gcs\  such that $\J_1$ and $\G$ commute.
\end{definition}

Given a \gcs\ $\J$, we let $L$ be its $+i$-eigenspace. We have that $\J \in \spin{\V}$, hence it decomposes $\wedge^{\bullet}V^*$ into the eigenspaces of its Lie algebra action on forms. The eigenvalues of $\J$ are of the form $ik$ with $-n\leq k\leq n$ and we denote the corresponding eigenspaces by $U^{k}_\J$ or simply $U^k$ if $\J$ is clear from the context.  For $v \in L$ and $\gf \in U^k$ we have that
$$\J(v\cdot \gf) = (\J v) \cdot \gf + v \cdot \J\gf  = i(k+1) v \cdot \gf,$$
that is, Clifford action of $L$  maps $U^k$ into $U^{k+1}$ and similarly Clifford action of $\overline{L}$ maps $U^k$ into $U^{k-1}$. Hence $U^n$ corresponds to the space of forms which annihilate $L$. Since $L$ is maximal isotropic in $\V\tensor \C$, $U^n$ is a line and it completely determines $\J$.  We call $U^n$ the {\it canonical line} of $\J$.

Given a generalized Hermitian structure, the automorphism  $\J_2 = \G \J_1$ is also orthogonal and squares to $-\Id$, hence it is a \gcs. Since  $\G$ and $\J_1$  commute they induce a decomposition of $\V_\C$ into intersections of their eigenspaces:
$$V_+^{1,0} = L_{1} \cap (V_+ \tensor \C),\qquad V_-^{1,0} = L_{1} \cap (V_- \tensor \C),$$
$$V_+^{0,1} = \bar{L_{1}} \cap (V_+ \tensor \C),\qquad V_-^{0,1} =\bar{L_{1}} \cap (V_- \tensor \C),$$
where $L_1$ is the $+i$-eigenspace of $\J_1$.

Similarly,  $\wedge^{\bullet}V^*_\C$ splits as the intersections of the eigespaces of $\J_1$ and $\J_2$: $U^{p,q} = U^p_{\J_1}\cap U^q_{\J_2}$ and since the Clifford action of $L_{1}$ changes the $p$-grading and the action of  $L_{2}$ changes the $q$-grading in specific ways,  the Clifford action of elements in $V_{\pm}^{1,0}$ and $V_{\pm}^{0,1}$ changes the $(p,q)$-grading  by $\pm 1$, as illustrated in Figure \ref{fig:Clifford action}.
\begin{center}
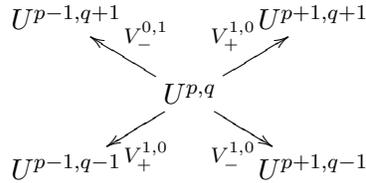
\begin{figure}[h!!]
$$\xymatrix@R=14pt@C=10pt{
& &U^{p-1,q+1} &  &U^{p+1,q+1} &&\\
&&                & \ar[lu]_{V_-^{0,1}}   \ar[ld]^{V_+^{1,0}}  U^{p,q} \ar[rd]_{V_-^{1,0}}\ar[ru]^{V_+^{1,0}}              &&& \\
& &U^{p-1,q-1} &  &U^{p+1,q-1} &&\\
}
$$
\caption{Action of elements of $V_{\pm}^{1,0}$ and $V_{\pm}^{0,1}$ on $U^{p,q}$.}\label{fig:Clifford action}
\end{figure}
\end{center}

Finally  in a generalized Hermitian manifold the generalized Hodge star is related to the action of $\JJ_i = e^{\frac{\pi \J_i}{2}}$, namely:

\begin{lemma}\label{lem:star} { (Gualtieri \cite{gualtieri-2004})} In a generalized Hermitian vector space one has
$$\star = -\JJ_1\JJ_2.$$
\end{lemma}

\section{The Nijenhuis tensor and integrability}

This section we transfer to manifolds the structures defined on vector spaces in Section \ref{sec:linear algebra}. In this context we work on a manifold with closed 3-form $(M,H)$, $H \in \Omega^3_{cl}(M)$. As before, the bundle $\TM$ is endowed with the natural pairing. Further, the space of sections of $\TM$  is endowed with the {\it Courant bracket}, the derived bracket corresponding to $d^H = d + H\wedge$:
$$\Cour{X+\xi,Y+\eta}_{d^H} =[X,Y] + \mc{L}_X \eta -i_Y d\xi - i_Y i_X H.$$
If the 3-form $H$ is clear from the context we denote this bracket simply by $\Cour{\cdot,\cdot}$.

\begin{definition}[Integrability conditions]~
\begin{itemize}
\item A {\it \gacs} is a smooth assignment of a \gcs\, $\J$ to each $T_pM$ for $p \in M$ and $\J$ is {\it integrable} if its $+i$-eigenbundle is involutive \wrt\ the Courant bracket, in which case we call $\J$ a {\it \gcs} on $M$.
\item A {\it generalized almost  Hermitian structure} is a smooth assignment of a generalized Hermitian structure, $(\G,\J_1)$, to each $T_pM$ and if $\J_1$ is integrable we call it  a {\it generalized Hermitian structure}.
\item A {\it \gks} is a  generalized Hermitian structure, $(\G,\J_1)$ such that $\J_2 = \G \J_1$ is also integrable.
\end{itemize}
\end{definition} 

Integrability of a \gcs\ can be determined by the action of $d^H$ on $\U^{k}$, the sheaf of sections of $U^k$. Indeed, we have the following characterization of  the behavior of $d^H$ on generalized almost complex manifolds:

\begin{theorem}[Cavalcanti \cite{cavalcantiskt}]\label{theo:Nijenhuis}
Let $\J$ be an almost \gcs\ and let $N$ be the Nijenhuis tensor of $\J$:
$$N:\tensor^3\Gamma(\bar{L}) \into \Omega^0(M;\C)\qquad N(X,Y,Z) = -2\IP{\Cour{X,Y},Z}.$$ 
Then $N \in \Gamma(\wedge^3 L)$,
$$d^H:\U^k \into \U^{k-3}\oplus \U^{k-1}\oplus\U^{k+1}\oplus\U^{k+3}$$
and the component of $d^H$ mapping $\U^k$ into $\U^{k+3}$ is the Clifford action of $N$ on forms. Similarly, the component mapping $\U^k$ into $\U^{k-3}$ is the action of $\bar{N}$ and they are both tensorial.
\end{theorem}
In particular we see from the above that involutivity of $L$ is equivalent to the vanishing of the Nijenhuis tensor which, in turn, furnishes the more usual integrability condition
\begin{equation}\label{eq:integrability}
d^H:\U^k \into \U^{k+1}\oplus \U^{k-1}.
\end{equation}
established by  Gualtieri  in \cite{gualtieri-2003}.

Of course, to determine the vanishing of the tensor $N$, or $\bar{N}$ for that matter, it is enough to show that $\bar{N}$ acts trivially in a space where the action of $\wedge^3 \bar{L}$ is faithful. For example, if
\begin{equation}\label{eq:canonical}
d^H:\U^n \into \U^{n-1},
\end{equation}
 then we conclude that $\bar{N}\equiv 0$ and the structure is integrable

If $(\G,\J_1)$ is a generalized almost  Hermitian structure, according to Theorem \ref{theo:Nijenhuis},  $d^H$ can not change either the `$p$' or  the `$q$' grading by more than three and it must switch parity. Hence $d^H$ decomposes as a sum of  eight operators and their complex conjugates
$$d^H =  \delta_++\deltabar_+ + \delta_-+\deltabar_- + N_+ +\bar{N_+}+ N_- + \bar{N_-} + N_1 +\bar{N_1}+ N_2 + \bar{N_2} + N_3 +\bar{N_3}+  N_4 +\bar{N_4};$$
\begin{equation}\label{eq:operators}
\begin{aligned}
\delta_+ :\U^{p,q} \into \U^{p+1,q+1},&\qquad \delta_-:\U^{p,q} \into \U^{p+1,q-1},\\
N_+ :\U^{p,q} \into \U^{p+3,q+3},&\qquad N_-:\U^{p,q} \into \U^{p+3,q-3},\\
N_1 :\U^{p,q} \into \U^{p-1,q+3},&\qquad N_2:\U^{p,q} \into \U^{p+1,q+3},\\
N_3 :\U^{p,q} \into \U^{p+3,q+1},&\qquad N_4:\U^{p,q} \into \U^{p+3,q-1}.
\end{aligned}
\end{equation}
and we can draw in a diagram all the possible nontrivial components of $d^H|_{\U^{p,q}}$ as arrows (see Figure \ref{fig:nonintegrable}).
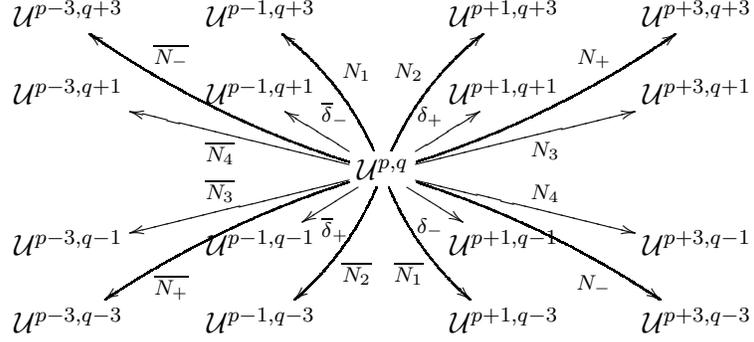
\begin{figure}[h!!]
\begin{center}
$$\xymatrix@R=14pt@C=10pt{
\U^{p-3,q+3}& &\U^{p-1,q+3} &  &\U^{p+1,q+3} &&\U^{p+3,q+3}\\
\U^{p-3,q+1}& &\U^{p-1,q+1} &  &\U^{p+1,q+1} &&\U^{p+3,q+1}\\
&&                &\ar@/ _0.5pc/[rrruu]^(0.7){N_+} \ar@/ ^0.5pc/[rrrdd]_(0.7){N_-} \ar@/ _0.5pc/[luu]_{N_1} \ar@/ ^0.5pc/[ruu]^{N_2}  \ar[rrru]_{N_3} \ar[rrrd]^{N_4}  \ar[lu]_{\deltabar_-}   \ar[ld]^{\deltabar_+} \U^{p,q}  \ar[ru]^{\delta_+}  \ar[rd]_{\delta_-}  \ar@/ _0.5pc/[rdd]_{\bar{N_1}}   \ar@/ ^0.5pc/[ldd]^{\bar{N_2}} \ar[llld]_{\bar{N_3}} \ar[lllu]^{\bar{N_4}}   \ar@/ _0.5pc/[llldd]^(0.7){\bar{N_+}} \ar@/ ^0.5pc/[llluu]_(0.7){\bar{N_-}}        &&& \\
\U^{p-3,q-1}& &\U^{p-1,q-1} &  &\U^{p+1,q-1} &&\U^{p+3,q-1}\\
\U^{p-3,q-3}& &\U^{p-1,q-3} &  &\U^{p+1,q-3} &&\U^{p+3,q-3}\\
}
$$
\caption{Representation of the nontrivial components of $d^H$ when restricted to $\U^{p,q}$ for an generalized almost Hermitian structure.}\label{fig:nonintegrable}
\end{center}
\end{figure}

\begin{definition}
The tensors $N_\pm$ and $N_{i}$, $i=1,2,3,4$ are the {\it intrinsic torsion} of the generalized almost Hermitian structure $(\G,\J_1)$.
\end{definition}

Integrability of $\J_1$ implies that $d^H$ only changes the `$p$' grading by $\pm 1$. Yet, since $\J_2$ is not integrable, $d^H$ can change the $q$ degree by $\pm1$ and $\pm 3$, hence in a generalized Hermitian manifold, $d^H$ decomposes into eight components and the Nijenhuis tensor of $\J_2$ decomposes in two components  ($N_1$ and $N_2$) as shown  in  Figure \ref{fig:genhermitian}.
\begin{center}
\begin{figure}[h!!]
$$\xymatrix@R=14pt@C=10pt{
& &\U^{p-1,q+3} &  &\U^{p+1,q+3} &&\\
& &\U^{p-1,q+1} &  &\U^{p+1,q+1} &&\\
&&                & \ar@/ _0.5pc/[luu]_{N_1} \ar@/ ^0.5pc/[ruu]^{N_2}   \ar[lu]_{\deltabar_-}   \ar[ld]^{\deltabar_+} \U^{p,q}  \ar[rd]_{\delta_-}\ar[ru]^{\delta_+}     \ar@/ _0.5pc/[rdd]_{\bar{N_1}}   \ar@/ ^0.5pc/[ldd]^{\bar{N_2}}         &&& \\
& &\U^{p-1,q-1} &  &\U^{p+1,q-1} &&\\
& &\U^{p-1,q-3} &  &\U^{p+1,q-3} &&\\
}
$$
\caption{Decomposition of $d^H$  in a generalized Hermitian manifold.}\label{fig:genhermitian}
\end{figure}
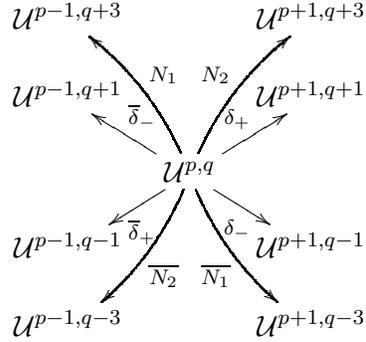
\end{center}

Integrability of $\J_2$ corresponds to the vanishing of  the Nijenhuis tensor of $\J_2$, hence for a generalized K\"ahler manifold we can decompose $d^H$ as a sum of four operators: $\delta_{\pm}$ and  their complex conjugates, as pictured in Figure \ref{fig:gkdecomposition}.
\begin{center}
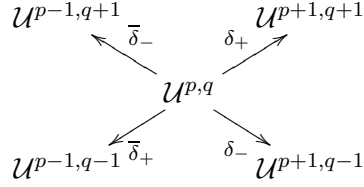
\begin{figure}[h!!]
$$\xymatrix@R=14pt@C=10pt{
& &\U^{p-1,q+1} &  &\U^{p+1,q+1} &&\\
&&                & \ar[lu]_{\deltabar_-}   \ar[ld]^{\deltabar_+}  \U^{p,q} \ar[rd]_{\delta_-}\ar[ru]^{\delta_+}              &&& \\
& &\U^{p-1,q-1} &  &\U^{p+1,q-1} &&\\
}
$$
\caption{Decomposition of $d^H$  in a generalized K\"ahler manifold.}\label{fig:gkdecomposition}
\end{figure}
\end{center}
Since $(d^H)^2=0$, we get 

\begin{corollary}\label{cor:hermitian relations}
In  a \gkm, 
\begin{equation}\label{eq:gkrelations}
\delta_{\pm}^2 =0, \qquad \{\delta_+,\delta_-\}=\{\delta_+,\deltabar_-\}=0\qquad \{\delta_+,\deltabar_+\} + \{\delta_-,\deltabar_-\} =0.
\end{equation} 
\end{corollary}

Using the inner product \eqref{eq:inner product} one can form the adjoints of the operators $\delta_{\pm}$ and then form the corresponding Laplacians, e.g., $\triangle_{\delta_+} = \delta_+\delta_+^* + \delta_+^*\delta_+$. Particularly relevant for the deformation problem are the results of Gualtieri regarding Hodge theory on a \gkm. Using Lemma \ref{lem:star} and integration by parts Gualtieri proved:

\begin{theorem}[Gualtieri \cite{gualtieri-2004}]\label{theo:Gualtieri-Hodge} In a generalized K\"ahler manifold
$$\delta_+^* = - \deltabar_+ \qquad \delta_-^* = \deltabar_-;$$
$$\triangle_{d^H} = 4\triangle_{\delta_+} = 4 \triangle_{\delta_-}= 4\triangle_{\deltabar_+} = 4 \triangle_{\deltabar_-}.$$
In paticular, the $d^H$-Laplacian preserves the spaces $\U^{p,q}$ and the $d^H$-cohomology of a compact \gkm, inherits a $\Z^2$-grading. 
\end{theorem}

\section{Deformations of  generalized K\"ahler structures}

In this section we present a proof of Goto's  theorem on stability of \gkss. The first thing to remark is that while the usual approach to deformation of \gcss\ and Lie bi-algebroids from  \cite{gualtieri-2003,MR1472888} is useful in identifying actual deformations of a \gcs, it does so by use of endomorphisms of $\TM$ which are not orthogonal and hence are not very well suited for the study of simultaneous deformations of two or more structures. The first lemma below, a simplified re-work of Proposition 2.6 in Goto's paper \cite{MR2669364} solves this problem. As before we let $\J_1$ be a generalized complex structure and $L_1$ be its $+i$-eigenspace.

\begin{lemma}\label{lem:make it orthogonal}
For each $p \in  M $ there is a disc centered at the origin in $(\wedge^2L_1\oplus \wedge^2\bar{L_1})^\R_p$, where  $(\cdot)^\R$ denotes the real elements in the vector space, such that any maximal isotropic of $\T_p M\tensor \C$ near $L_1|_p$ corresponds to the orthogonal action of $e^a$ on $L_1$ for a unique $a$ in such disc. 
\end{lemma}
\begin{proof}
Indeed, the space of \gcss\ of the same parity as $\J_1$  on  $T_pM$ is the homogeneous space $SO(\T_pM)/\mathrm{Stab}(\J_1) \cong SO(n,n)/U(n,n)$. Hence, composing the exponential map with the projection
$$\frak{so}(\T_p M)\into SO(\T_p M) \into SO(\T_p M)/\mathrm{Stab}(\J_1)$$
gives a submersion in a \nhood\ of $0$. Since the elements in $ \frak{so}(\T_p M)$ preserving $\J_1$ are those in $(L_1\tensor \bar{L_1})^\R$, and $(\wedge^2 L_1\oplus \wedge^2\bar{L_1})^{\R}$ is a complementary subspace, we have that
$$(\wedge^2 L_1\oplus \wedge^2\bar{L_1})^{\R} \into SO(\T_p M)/\mathrm{Stab}(\J_1).$$
is a local diffeomorphism. That is, for each small deformation of a \gcs\ $\J_1$ on $V$ there is a unique element $a \in  (\wedge^2 L_1\oplus \wedge^2\bar{L_1})^{\R}$ which realizes it.
\end{proof}

Observe that for $a \in \Gamma(\wedge^2 L_1\oplus \wedge^2\bar{L_1})^\R$, the deformed \gcs\ is given by
$$\J_a = e^{a}_*\J_1 e^{-a}_*.$$

\begin{theorem}[Goto \cite{MR2669364}]\label{theo:goto}
Let $(M,H)$ be a compact manifold and $(\J_1,\J_2)$ be a generalized K\"ahler structure on $M$ such that the canonical bundle of $\J_2$ admits a nowhere vanishing closed section $\psi$. Let $\J_{1t}$ be a family of deformations of the structure $\J_1$ determined by an analytic function $a: D \into \Gamma(\wedge^2 L_1\oplus \wedge^2\bar{L_1})^\R$, where $D$ is a disc around the origin in $\C$. Then there is an analytic family of deformations, $\J_{2t}$, of $\J_2$ determined by closed forms $\psi_t$ such that $\psi_0=\psi$ and $(\J_{1t},\J_{2t})$ is a \gks\ on $M$.
\end{theorem}
\begin{proof}
The basic idea of the proof is that we can pre-compose the deformation determined by $a$ by any automorphism of $\J_1$ as this does not change the final deformation of $\J_1$. That is, once $a_t \in \Gamma(\wedge^2 L_1\oplus \wedge^2\bar{L_1})^{\R}$ is chosen, we still have $\Gamma(L_1\tensor \bar{L_1})^\R$ worth of choices on how to change $\J_2$ so that the pair is a generalized Hermitian structure. The quest then is to find $b_t \in \Gamma(L_1\tensor \bar{L_1})^\R$ such that
$$\J_{2t} = e^{a_t}e^{b_t} \J_2  e^{-b_t} e^{-a_t}$$
is integrable. This is done by induction using a power series argument. Finally, to finish the proof one must show that the series obtained converges.

As we mentioned in the proof of Lemma \ref{lem:make it orthogonal}, the elements in $\frak{so}(\TM)$ whose exponential preserve $\J_1$ are those in $\Gamma(L_1\tensor \bar{L_1})^\R$, so, for {\it any} $b:D \into \Gamma(L_1\tensor \bar{L_1})^\R$ the pair $(\J_{1t},\J_{2t})=(e^{a_t}\J_1e^{-a_t},e^{a_t} e^{b_t}\J_2e^{-b_t} e^{-a_t})$ is a generalized Hermitian structure on $M$. Since the deformed structures are obtained from $(\J_1,\J_2)$ by the exponential action of an element in $\frak{so}(\TM) = \frak{spin}(\TM)$, the corresponding decompositions of forms are related by that same transformation:
\begin{equation}\label{eq:upq relation}
U^{p,q}_t = e^{a_t}e^{b_t}U^{p,q}.
\end{equation}

Since $\J_{1t}$ is integrable for {\it any} choice of $b$ we have that, for any choice of $b$, $d^H$ splits \wrt\ $(\J_{1t},\J_{2t})$ into eight components, as depicted in Figure \ref{fig:genhermitian}, and using the isomorphism \eqref{eq:upq relation} we also have that $e^{-b_t}e^{-a_t}d^He^{a_t} e^{b_t}$ splits in eight components \wrt\ the decomposition of forms induced by $(\J_1,\J_2)$.

The map $a$ is  analytic, say $a = \sum a_j t^j$, and  we will solve for $b$ given as a  series, $ b_t = \sum (\beta_j +\overline{\beta_j})t^j$, with  and $\beta_j \in \Gamma(\wedge V^{0,1}_+\tensor V_-^{1,0})$.  Our task is to find  $b_t$ such that $d^H\psi(t) =0$. The requirement that $\psi(0)=\psi$ forces us to chose $\beta_0=0$ and with this choice we have that $d^H\psi(t)|_{t=0} = d^H\psi =0$, i.e., $d^H\psi(t)$ vanishes to order zero.

Assume by induction that we have chosen  $\beta_j$  for $j <k$ such that, as a function of $t$, $d^He^ae^{(b)_{<k}} \psi$ vanishes to order $k-1$, where $b_{<k} = \sum_{j<k} ( \beta_j + \bar{\beta_j})$ and now we choose $\beta_k$ so that $d^He^ae^{b_{<k+1}} \psi$ vanishes to order $k$. Let $\Rho$ denote the order $k$ term of $d^He^ae^{b_{<k+1}} \psi$:
$$\Rho = d^H(e^ae^{b_{<k+1}} \psi)_k.$$
Since $d^He^ae^{b_{<k+1}} \psi$ vanishes to order $k-1$,  and $(e^{-b_{<k+1}}e^{-a}-1)$ vanishes to order zero, we see that $\Rho$ is the same as the order $k$ term of $e^{-b_{<k+1}}e^{-a}d^He^ae^{b_{<k+1}} \psi$, in particular, from the description of the decomposition of $d^H$ on a generalized Hermitian manifold we conclude that  $\Rho \in \U^{1,n-1}\oplus \U^{1,n-3}\oplus \U^{-1,n-1}\oplus \U^{-1,n-3}$.

For any choice of $\beta_k$ we have that
\begin{equation}\label{eq:rho}
\Rho = d^H(\beta_k \psi) + \rho(a_1,\cdots,a_{k},b_1,\cdots,b_{k-1}),
\end{equation}
where $\rho$ takes values in $\U^{1,n-1}\oplus \U^{1,n-3}\oplus \U^{-1,n-1}\oplus \U^{-1,n-3}$:
$$ \rho = \rho^{1,n-1} +  \rho^{1,n-3} + \rho^{-1,n-1}+\rho^{-1,n-3}, \qquad \mbox{with }\rho^{p,q} \in  \U^{p,q},$$
since both $\Rho$ and $d^H(\beta_k\psi)$ lie in these spaces. Also $\rho$ is $d^H$-exact since both $\Rho$ and $d^H\beta_k$ are. So, in order to complete the inductive step, we must show that we can choose  $\beta_k$ such that
$$ d^H(\beta_k  \psi) = -\rho(a_1,\cdots,a_{k},b_1,\cdots,b_{k-1}).$$
Finally, any element in $\U^{0,n-2}$ is of the form $\beta\psi$ for some $\beta \in \Gamma(V_-^{1,0}\tensor V_+^{0,1})$, so finding $\beta_k$  is equivalent to finding $\gf \in \U^{0,n-2}$ such that $\rho = d^H \gf$.
\begin{center}
\begin{figure}[h!!]
$$\xymatrix@R=14pt@C=10pt{
& &\rho^{-1,n-1}\in \U^{-1,n-1} &  &\rho^{1,n-1} \in \U^{1,n-1} &&\\
& &                &  \ar[ld]_{\deltabar_+}  \ar[lu]^{\deltabar_-} \gf\in \U^{0,n-2} \ar[rd]^{\delta_-}  \ar[ru]_{\delta_+} &&& \\
& &\rho^{-1,n-3} \in \U^{-1,n-3} &  &\rho^{1,n-3} \in \U^{1,n-3} &&\\
}
$$
\caption{At the heart of the problem we have that $\rho\in \U^{1,n-1}\oplus \U^{1,n-3}\oplus \U^{-1,n-1}\oplus \U^{-1,n-3}$ is exact and must show that it is in $d^H(\U^{0,n-2})$.}\label{fig:the problem}
\end{figure}
\end{center}

Since  $\rho$ is $d^H$-exact, we have that $\rho = \triangle G \rho$, where $\triangle$ is the Laplacian of any of the operators $\delta_\pm$, $\deltabar_\pm$ and $G$ the corresponding Green's operator, which due to Theorem \ref{theo:Gualtieri-Hodge} does not depend on which of the four operators is used. Since $\triangle$ and $G$ preserve the spaces $\U^{p,q}$, we have that individually $\rho^{p,q} =  \triangle G \rho^{p,q}$. Further, the condition $d^H \rho =0$, among other things, implies the  following:
\begin{IEEEeqnarray}{c}
 \delta_+\rho^{1,n-1} =\deltabar_+\rho^{-1,n-3} = \delta_-\rho^{1,n-3}  = \deltabar_- \rho^{-1,n-1} = 0,\label{eq:eq0}\\
\delta_-\rho^{-1,n-1} + \deltabar_-\rho^{1,n-3} + \delta_+\rho^{-1,n-3}+ \deltabar_+\rho^{1,n-1}  =0.\label{eq:eq5}
\end{IEEEeqnarray}

Now, let
\begin{equation}\label{eq:phi}
\gf = G(\delta_- \rho^{-1,n-1}+ \deltabar_-  \rho^{1,n-3} )=-G(\delta_+ \rho^{-1,n-3} + \deltabar_+  \rho^{1,n-1}) \in \U^{0,n-2},\\
\end{equation}
where the identity for the two expressions for $\gf$ follows from \eqref{eq:eq5}. Then we compute the different components of $d^H\gf$. We start with the $\U^{-1,n-3}$-component, which is given by $\deltabar_+\gf$:
$$\deltabar_+\gf =- G(\deltabar_+\delta_+ \rho^{-1,n-3}) = - G\circ (\deltabar_+\delta_+ + \delta_+ \deltabar_+)(\rho^{-1,n-3}) = G \triangle_{\delta_+}\rho^{-1,n-3} =\rho^{-1,n-3},$$
where in the first equality we used the second expression for $\gf$, in the second equality we used \eqref{eq:eq0} and in the third and fourth we used Theorem \ref{theo:Gualtieri-Hodge}. The remaining components follow the same paradigm and we get  $\rho = d^H \gf$, which completes the induction step.

Proof of convergence uses standard elliptic estimate arguments and is done along the same lines of Kodaira and Spencer's original argument for deformations of complex structures (c.f. Section 5.3 (c) in \cite{MR815922}). The main points of the argument being that one can (inductively) bound the $(l-1,\alpha)$-Holder norm of the function $\rho$ from \eqref{eq:rho},  by the $(l,\alpha)$-Holder norm of the functions $a_i$ for $i\leq k$, as $\rho$ depends  on $a_i$, $b_i$ and their first derivative.   Hence, due the smoothing properties of the Green operator, the $(l,\alpha)$-norm of $\gf$ defined in \eqref{eq:phi} (and consequently of $b_k$) is also bounded by the $(l,\alpha)$-norm of the functions $a_i$ for $i\leq k$. Convergence of $\sum a_i t^i$ then implies convergence of $\sum b_i t^i$ in a possibly smaller radius.

\end{proof}

\providecommand{\bysame}{\leavevmode\hbox to3em{\hrulefill}\thinspace}
\providecommand{\MR}{\relax\ifhmode\unskip\space\fi MR }
\providecommand{\MRhref}[2]{%
  \href{http://www.ams.org/mathscinet-getitem?mr=#1}{#2}
}
\providecommand{\href}[2]{#2}

\end{document}